\documentclass[final,hidelinks,onefignum,onetabnum]{siamart250211}
\pdfoutput=1
\usepackage[utf8]{inputenc}
\usepackage{enumitem}

\usepackage{lipsum}
\usepackage{amsfonts}
\usepackage{graphicx}
\usepackage{epstopdf}
\usepackage{algorithmic}
\ifpdf
  \DeclareGraphicsExtensions{.eps,.pdf,.png,.jpg}
\else
  \DeclareGraphicsExtensions{.eps}
\fi

\newsiamremark{remark}{Remark}
\newsiamremark{hypothesis}{Hypothesis}
\crefname{hypothesis}{Hypothesis}{Hypotheses}
\newsiamthm{claim}{Claim}
\newsiamremark{fact}{Fact}
\crefname{fact}{Fact}{Facts}

\headers{Well-Posedness for the Rosenzweig-MacArthur Predator-Prey Model with Internal Stochasticity}{Shuo Wang, and Jiguang Yu}

\author{Louis Shuo Wang\thanks{Department of Mathematics, University of Tennessee, Knoxville, TN 37996, USA 
  (\email{swang116@vols.utk.edu}).}
\and Jiguang Yu\thanks{Department of Mathematics, University College London, Gower Street, London, WC1E  6BT, United Kingdom 
  (\email{zcahyuc@ucl.ac.uk}).}}

\usepackage{amsopn}

\title{Well-Posedness for the Rosenzweig-MacArthur Model 
with Internal Stochasticity}

\ifpdf
\hypersetup{
  pdftitle={Well-Posedness for the Rosenzweig-MacArthur Model 
with Internal Stochasticity},
  pdfauthor={Louis Shuo Wang, and Jiguang Yu}
}
\fi

\IfFileExists{ex_supplement.aux}{%
  \externaldocument[][nocite]{ex_supplement}%
}{}

\begin{document}
\pagestyle{empty}
\maketitle
\begin{abstract}
In this work, we propose a stochastic version of the Rosenzweig-MacArthur model solely driven by internal demographic noise, extending classical Lotka-Volterra-type systems focused on external noise. We give a criterion for the existence and uniqueness of autonomous stochastic differential equations (SDEs) on an open submanifold of $\mathbb{R}^{n}$, and the framework allows for a wider choice of Lyapunov functions. In the meantime, the invariance of open submanifolds, which is a biologically feasible result and has been implicitly incorporated into many biological and ecological models, facilitates the application of analytic tools typically suited to $\mathbb{R}^{d}$ and indicates the persistence of predator and prey populations, thus providing a criterion for determining whether a population will become extinct. We apply the well-posedness criterion to our stochastic Rosenzweig-MacArthur model and show the existence and uniqueness of solutions. Furthermore, the asymptotic estimates of solutions are obtained, indicating the at most exponential growth of the population with internal stochasticity. Some numerical experiments are performed, which illustrate the discrepancy between the deterministic and stochastic models. Overall, this work demonstrates the broad applicability of our results to ecological models with constrained dynamics, offering a foundation for analyzing extinction, persistence, and well-posedness in systems where internal randomness dominates. This paper not only promotes the development of stochastic modeling and stochastic differential equations in theoretical ecology but also proposes a rigorous mathematical methodology for studying the predator-prey system with internal stochasticity.
\end{abstract}



\section{Introduction}
The classical Predator-Prey Lotka-Volterra model is the following: 
\begin{align}\label{eq1.1}
    \begin{cases}
        \dfrac{dN}{dt}=N(a-bP), \\[6pt]
        \dfrac{dP}{dt}=P(cN-d),
    \end{cases}
\end{align}
which was proposed by Volterra for the predation of one species (the prey $N$) by another species (the predator $P$) to explain the oscillatory behaviors of certain fish catches in the Adriatic 
\cite{murray2007mathematical}. The parameters $a,b,c,d$ in the model are all positive; the term $aN$ describes the intrinsic growth of the prey in the absence of the predator, 
and the term $dP$ represents the death rate of the predator population in the absence of the prey. 
The predation effect of the prey feeding the predator follows from the cross term $NP$,
which is proportional to the available prey resource, the size of the predation population, and can be regarded as the conversion of energy from the prey to the predator:
$bNP$ from the prey and $cNP$ to the predator. 

Although simple, the above model is the starting point for more general and complex models to characterize periodic and oscillatory behaviors in biological and ecological systems, 
such as the interaction of population, chemical reactions, genetic evolution, and many other phenomena in life sciences
\cite{arnold2012stochastic,crow2017introduction,goel1971volterra,ludwig1975persistence,may2019stability,rescigno1973deterministic}. Following from it, many more realistic models have been proposed, 
incorporating more detailed factors in the systems under consideration. One of them is the Rosenzweig-MacArthur predator-prey model:
\begin{align}\label{eq1.2}
    \begin{cases}
        \dfrac{dN}{dt}=rN(1-\dfrac{N}{K})-\dfrac{sNP}{1+s\tau N}, \\[6pt]
        \dfrac{dP}{dt}=-cP+d \dfrac{sNP}{1+s\tau N},
    \end{cases}
\end{align}
where $r,K>0$ and $c,d>0$. $h(N)=\dfrac{sN}{1+s\tau N} $ is the number of prey caught per predator per unit of time, or the per predator kill rate,
and has units of $1/\mbox{time}$, where $s$ is the predator search rate for prey in units of area per unit time, and $\tau $ is the time it takes for the predator to handle (stalk, catch, eat, burp, clean teeth, sleep) the prey before the predator can search for another prey.    
After nondimensionalization, we arrive at the following system:
\begin{align}\label{eq1.3}
\begin{aligned}
    \begin{cases}
\dfrac{dN}{dt} &= N \left( 1-\dfrac{N}{k}  \right) - \dfrac{mNP}{1+N}, \\[6pt]
\dfrac{dP}{dt} &= -cP + \dfrac{mNP}{1+N}.     
    \end{cases}
\end{aligned}
\end{align} 
Furthermore, the non-dimensionalization process can be found in \cref{app1}.

The necessity of randomizing classical deterministic equations and therefore considering stochastic models stems from external environmental random fluctuations and internal randomness
caused by random fluctuations in the growth, reproduction and deaths of individuals in the population \cite{abundo1991stochastic}. The other reason to use the stochastic model is that we cannot explicitly capture all the necessary factors that potentially impact the system under our consideration, and we use stochasticity to model these unknown factors. 
Thus, stochastic models are widespread in biological and ecosystem models and are used to capture random fluctuations implicitly. Based on the above reasoning, we modify (\ref{eq1.3}) and add Brownian motions to the model (\ref{eq1.3}) to characterize the variations, or white noises, in the density of prey $N$ and predator $P$ at any instant $t$:
\begin{align}\label{eq1.4}
    \begin{cases}
        dN&=\underbrace{\left(N \left( 1-\dfrac{N}{k}  \right) - \dfrac{mNP}{1+N} \right)dt}_{I_{1}}
        +\underbrace{\left[N \left( 1+\dfrac{N}{k}  \right) + \dfrac{mNP}{1+N} \right]^{1/2}dB_{t}^{1}}_{I_{2}},\\[6pt]
        dP&=\underbrace{\left(-cP + \dfrac{mNP}{1+N} \right)dt}_{I_{3}}
        +\underbrace{\left[cP + \dfrac{mNP}{1+N} \right]^{1/2}dB_{t}^{2}}_{I_{4}},
    \end{cases}
\end{align}
where $B_{t}^{i}$, $1\leq i\leq 2$ are independent standard Brownian motions. 
In this model, the drifts are given by the terms $I_{1}$ and $I_{3}$, 
while the deviations are given by the terms 
$I_{2}$ and $I_{4}$.  
Variations, which are the square of deviations, are roughly proportional to the total population density $(N+P)$, and this means that, in a larger population, there are more variations in the population density. 
Biologically, a larger population harbours more internal randomness. 
(Note that we only consider the internal randomness in our model (\ref{eq1.4})). 

This paper is organized as follows: \cref{sec:novelty} presents the theoretical contributions and novelty of our work.
We give an introduction to the dynamical properties of the deterministic Rosenzweig-MacArthur predator-prey model in \cref{sec:dynamical},
and prove the existence and uniqueness of solutions to the stochastic Rosenzweig-MacArthur predator-prey model in \cref{sec:stochastic}.
Some numerical results are presented in \cref{sec:numerical}, and we conclude the paper in \cref{sec:conclusion}.

\section{Theoretical Contributions and Novelty}\label{sec:novelty}
First, we establish the existence and uniqueness of solutions for the stochastic model on a constrained domain (an embedded submanifold) rather than the full Euclidean space.
Extending the existence and uniqueness to an embedded submanifold (an open subset of $\mathbb{R}^n$ in our work) is nontrivial. 
It ensures that the model is well-posed and feasible in a biological sense. 
In classic stochastic analysis, solutions in $\mathbb{R}^{n}$ are generally assumed, and SDEs can be formulated on the whole space. 
However, the case where the state variables must satisfy nonlinear constraints (e.g., state variables must be nonnegative or positive) is not addressed. 
Our work appears relatively novel compared to other SDE literature: while standard theorems guarantee solutions of some specific SDEs (these with global Lipschitz continuous coefficients) exist 
and are unique in $\mathbb{R} ^{n}$, we give a criterion (Theorem~\ref{thm4.3}) for the existence and uniqueness of solutions to a special system of SDEs, 
whose coefficients are independent of the time variable. 
Especially, our new criterion applies to many Lokta-Volterra-type problems, showing the significance and applications of our current work.
This criterion involves a Lyapunov function, and in the application to our stochastic Rosenzweig-MacArthur predator-prey model, 
we construct a suitable Lyapunov function adapted to the specific form of internal noise and interaction of the model, 
facilitating the proof of the existence and uniqueness of solutions.
This Lyapunov function is a key tool in our work and can be directly applied to many other Lotka-Volterra-type problems.

Second, we also prove the invariance of the domain, which is an open submanifold of $\mathbb{R} ^{n}$ in our work.
This invariance result enables us to utilize many estimates 
which require some conditions on the whole space $\mathbb{R} ^{n}$ to our model defined on an open submanifold of $\mathbb{R} ^{n}$. 
In stochastic theory, one typically checks that the drift and diffusion terms do not push trajectories outside a prescribed domain, because ensuring that
predator and prey variables never become negative is crucial to biological realism. 
Although many SDE population models assume invariance heuristically, we prove it rigorously here.

Overall, we based on the classical Rosenzweig-MacArthur predator-prey model and made a combination and application of these tools (generalised existence and uniqueness result on a submanifold, Lyapunov function, and domain invariance) 
to a stochastic version model driven only by internal randomness. 
In the existing literature on stochastic modelling, the existence and uniqueness of SDE systems with external randomness have been established, and extinction or persistence results are well-established. 
However, extending them to ensure rigorous behaviour of an intrinsically stochastic predator-prey system is a new development. 

\section{Dynamical Properties}
\label{sec:dynamical}

In this section, some concepts of dynamical systems are first introduced. 
Readers who are already familiar with dynamical systems can skip this part. The ultimate goal is to apply the methods in dynamical systems to study the possible dynamics of our biological system. 
We assume $f\in C^{1}(\mathbb{R}^{n};\mathbb{R}^{n})$ and $f$ satisfies the global Lipschitz condition: 
there exists a positive number $M>0$ such that 
\begin{align*}
    \left\lVert f(x)-f(y) \right\rVert \leq M \left\lVert x-y \right\rVert 
 \end{align*} for all $x,y \in \mathbb{R}^{n}$, where $\left\lVert \cdot  \right\rVert : x=(x_{1},\ldots,x_{n})\in \mathbb{R}^{n} \mapsto \left(\sum\limits_{k=1}^{n} x_{k}^{2}\right)^{1/2}$ is the standard norm, 
 or distance function, on $\mathbb{R}^{n}$. For a dynamical system $D$ characterized by an autonomous system 
\begin{align}\label{eq3.1}
\dot{x}=f(x)
\end{align}
that is defined on $\mathbb{R}^{n}$, the $C^{1}$ flow of $D$ is a $C^{1}$ map $\phi:\mathbb{R}\times \mathbb{R}^{n} \to \mathbb{R}^{n}$ that satisfies 
\begin{align}\label{eq3.2}
&\phi(0,x) = x\quad \quad \quad \quad \quad \quad \,\,\,\,\,\,\,\, \mbox{ for any } x \in \mathbb{R}^{n}, \nonumber \\[6pt] 
&\phi(t,\phi(s,x)) = \phi(t+s,x) \quad \,\, \mbox{ for all } s,t \in \mathbb{R}, x \in \mathbb{R}^{n}, \\[6pt] 
&\dfrac{d}{dt}|_{t=t_{0}} \phi(t,x_{0}) = f(\phi(t_{0},x_{0})). \nonumber
\end{align}
The flow $C^{1}$ generally does not exist for a general continuous vector field $f$ on $\mathbb{R}^{n}$, 
but exists for such a continuously differentiable and globally Lipschitz continuous function \cite{perko2013differential}, 
and this justifies our definition for the $C^{1}$ flow. For such a smooth vector field $f$, 
it can also be shown that, for every $x_{0}\in \mathbb{R}^{n}$, the initial value problem (\ref{eq3.1}) 
with $x(0)=x_{0}$ has a unique solution $x(t)$ for all $t\in \mathbb{R}$. 
Furthermore, from the definition of the $C^{1}$ flow, for any given $x_{0}\in \mathbb{R}^{n}$, 
this is a unique solution $x(t)=\phi(t,x_{0})$ to the above initial value problem (\ref{eq3.1}) and $x(0)=x_{0}$. 
Intuitively, $\phi(t,x_{0})$ is the position where $x_{0}$ at $t=0$ is moved forward at time $t$ by the vector field $f$, and this is similar to a particle movement in the river. 
Thus, $x(t)=\phi(t,x_{0})$ for $t\in \mathbb{R}$ defines a trajectory, or an orbit, of (\ref{eq3.1}) through $x_{0}$.
For a solution $x(t)$ of (\ref{eq3.1}) with the initial condition $x(0)=x_{0}$, we define the trajectory, 
or orbit of $\phi (t,x_{0})$ as $\Gamma(\phi (\cdot, x_{0}))=\left\{ \phi (t,x_{0}):t\in \mathbb{R} \right\} $, 
and the positive semi-trajectory of $\phi (t, x_{0})$ as $\Gamma ^{+}(\phi (\cdot, x_{0}))=\left\{ \phi (t,x_{0}): t\geq 0 \right\} $.

A point $p\in \mathbb{R}^{n}$ is called an $\omega$-limit point of the trajectory $\Gamma(\phi(\cdot,x_{0}))$ 
if there is a sequence of time $\{t_{n}\geq 0\}_{n=1}^{\infty}$ with 
$\underset{n\rightarrow \infty}{\lim} t_{n}=\infty$ such that 
\begin{align*}
    \underset{n\rightarrow \infty}{\lim}\phi(t_{n},x_{0})= p
\end{align*}
The set of all $\omega$-limit points of the trajectory $\Gamma(\phi(\cdot,x_{0}))$ is called 
the $\omega$-limit set of the trajectory $\phi(\cdot,x_{0})$ 
and is denoted by $\omega \left( \Gamma ^{+}(\phi (\cdot, x_{0})) \right) $. 
The $\alpha$ limit point and the $\alpha$-limit set of a trajectory are defined similarly, 
except that we require $t_{n}\rightarrow -\infty$. 
The $\alpha $-limit set and $\omega $-limit set is a characterization of the limit behavior, 
or long-time behavior, of the trajectory of the solution; e.g., if the 
$\omega $-limit set of a trajectory $\Gamma (\gamma )$ is a singleton $\left\{ p \right\} $, then we have \begin{align*}
\underset{t\to \infty }{\lim}\,\gamma (t)=p,
 \end{align*}  
and this shows that $p$ is the eventual limit of $\Gamma (\gamma )$ and a steady state (an equilibrium point) 
of the dynamical system (\ref{eq3.1}). Meanwhile,
if the singleton consisting of a critical point $p$ is the $\omega $-limit set of every solution
of the system (\ref{eq3.1}) (a point $p\in \mathbb{R} ^{n}$ is said to be a critical point of (\ref{eq3.1}) if $f(p)=0$), 
then the trajectory of every solution of the system (\ref{eq3.1}) will eventually approach $p$, and this shows that the critical point $p$ is globally asymptotically stable for this system (\ref{eq3.1}). 
See \cite{haddad2008nonlinear} for an introduction to the Lyapunov stability theory. 

In the end, we could speak of a more general dynamical system defined by some function $f$ between two
manifolds $M$ and $N$, and the concept of flow is borrowed from manifold theory. This function $f$ determines a vector field (still denoted by $f$, which will not cause any confusion) on the domain manifold $M$. Each solution of this dynamical system is in one-to-one correspondence with the integral curve of the vector field $f$, since in each local coordinates on the manifold to find the integral curve is the same as finding the solution to differential equations.
In the context of mainfolds, the flow $\theta $ for a vector field is defined similarly to (\ref{eq3.2}) except that $f$ is a map between two manifolds $M$ and $N$, and the complete space-time space
$\mathbb{R}\times \mathbb{R} ^{n} $ in the definition (\ref{eq3.2}) of the flow is replaced by an open subset $\mathcal{D}$ of $\mathbb{R}\times \mathbb{R} ^{n} $. 
To be more precise, the above definition defines a local flow for the vector field $f$ on $\mathcal{D}$, and we can extend the domain $\mathcal{D}$ of the flow to the maximal domain 
$\mathcal{D}^{\prime} $ and the extended flow $\theta ^{\prime} $ defined on $\mathcal{D}^{\prime} $ is called the maximal flow for the vector field $f$ ($\mathcal{D}^{\prime} $ is the maximal domain means 
we can not extend $f$ to a larger domain $\mathcal{D}^{\prime \prime}$ containing $\mathcal{D}^{\prime} $, and this is mainly because the solution of $f$ will blow up at the boundary of $\mathcal{D}^{\prime} $ 
so we can not extend the solution further). The fundamental theorem of flows \cite{lee2012smooth} guarantees a one-to-one correspondence between vector fields $f$ and maximal flows $\theta $. 
Intuitively, the flow $\theta $ is like a collection of solutions or integral curves for the vector field $f$, and the vector field $f$ is the infinitesimal generator of the flow $\theta$; 
e.g., a good approximation to an integral curve can be obtained by composing many small straight-line motions, with the direction and length of each motion determined by the value of the vector field at the point arrived at in the previous time step, and a flow can be thought of as a sequence of infinitely many infinitesimally small linear steps \cite{lee2012smooth}.
Thus, from this point of view, the study of various types of differential equations, such as ODE and PDE, can be unified by studying the vector fields or dynamical systems in manifolds.                                              

The concept of stability of a dynamical system was introduced by A.M. Lyapunov in 1892 \cite{liapounoff1907probleme,lyapunov1992general}, and we will introduce some concepts of stability of a deterministic dynamical system in this section.
Without loss of generality, we assume that $x=0$ is a critical point of $f$, then $x(t)\equiv 0$ is a trivial solution to (\ref{eq3.1}). 
\begin{definition}[\textbf{Lyapunov Stability}]\label{def3.1}\cite{haddad2008nonlinear}
    \begin{enumerate}[label=(\roman*)]
    \item The zero solution $x(t)\equiv 0$ is Lyapunov stable, if for all $\varepsilon >0$, there exists $\delta =\delta (\varepsilon )>0$ such that if $\left\lVert x(0) \right\rVert <\delta $, then $\left\lVert x(t) \right\rVert <\varepsilon  $ for $t\geq 0$, and the zero solution is unstable if it is not Lyapunov stable.          
    \item The zero solution $x(t)\equiv 0$ is (locally) asymptotically stable if it is Lyapunov stable and there exists $\delta >0$ such that if $\left\lVert x(0) \right\rVert <\delta $, then $\underset{t\to \infty }{\lim}\, x(t)=0$.
    \item The zero solution $x(t)\equiv 0$ is (locally) exponentially stable if there exist a positive constant $\alpha, \beta $, and $\delta $ such that if $\left\lVert x(0) \right\rVert <\delta $, then $\left\lVert x(t) \right\rVert \leq \alpha \left\lVert x(0) \right\rVert e^{-\beta t} $ for $t\geq 0$. 
    \item The zero solution $x(t)\equiv 0$ is globally asymptotically stable if it is Lyapunov stable and for all $x(0)\in \mathbb{R} ^{n}$, $\underset{t\to \infty }{\lim}\,x(t)=0$.
    \item The zero solution $x(t)\equiv 0$ is globally exponentially stable if there exist positive constants $\alpha $ and $\beta $ such that $\left\lVert x(t) \right\rVert \leq \alpha \left\lVert x(0) \right\rVert e^{-\beta t},\,t\geq 0 $ for all $x(0)\in \mathbb{R} ^{n}$.                
    \end{enumerate}
\end{definition}  
Exponential stability implies asymptotic stability, which in turn implies Lyapunov stability.
The Lyapunov direct method uses the Lyapunov function, which characterizes the distance of the solution and the zero solution, to determine the stability of
the zero solution \cite{haddad2008nonlinear}. Let $V:\mathcal{D}\to \mathbb{R} $ be a continuously differentiable function and $\phi (t,x)$ be the unique solution to (\ref{eq3.1}) 
that passes through $x$ at $t=0$, then the derivative of $V$ along the trajectory $\phi (t,x)$ is defined as 
$\dot{V}(x)=\dfrac{d}{dt}|_{t=0}V(\phi (t,x)) = V ^{\prime} (x)f(x)$, and it follows that
if $\dot{V}(x)$ is negative, then $V(x)$ decreases along the trajectory $\phi (t,x)$.        
\begin{theorem}[\textbf{Lyapunov Direct Method}]\label{thm3.2} \cite{haddad2008nonlinear}
    Consider the nonlinear dynamical system (\ref{eq3.1}) and assume that there exists a continuously differentiable function $V:\mathcal{D}\to \mathbb{R} $ such that 
    \begin{enumerate}[label= (\roman*)]
    \item $V(0)=0$,    
    \item $V(x)>0$ for $x\in \mathcal{D}\setminus \{0\}$, 
    \item $\dot{V}(x)\leq 0$ for $x\in \mathcal{D}$.   
    \end{enumerate} 
    Then the zero solution is Lyapunov stable. If, in addition, 
    \begin{enumerate}[label= (\roman*)]
    \item $\dot{V}(x)<0$ for $x\in \mathcal{D}\setminus \{0\}$,     
    \end{enumerate}
    then the zero solution is asymptotically stable. Finally, if there exists scalars $\alpha ,\beta ,\varepsilon >0$, and $p\geq 1$, such that $V:\mathcal{D}\to \mathbb{R}$ satisfies \begin{enumerate}[label= (\roman*)]
    \item $\alpha \left\lVert x \right\rVert ^{p}\leq V(x)\leq \beta \left\lVert x \right\rVert ^{p} $, for $x\in \mathcal{D}$,      
    \item $\dot{V}(x)\leq -\varepsilon V(x)$, for $x\in \mathcal{D}$,  
    \end{enumerate}   
    then the zero solution is exponentially stable.
\end{theorem} 
\begin{theorem}[\textbf{Lyapunov Direct Method}]\label{thm3.3} \cite{haddad2008nonlinear}
    Consider the nonlinear dynamical system (\ref{eq3.1}) and assume there exists a continuously differentiable function $V:\mathbb{R} ^{n}\to \mathbb{R} $ such that 
    \begin{enumerate}[label=(\roman*)]
    \item $V(0)=0$,    
    \item $V(x)>0$ for $x\in \mathbb{R} ^{n}\setminus \{0\}$,
    \item $\dot{V}(x)<0$ for $x\in \mathbb{R} ^{n}\setminus \{0\}$, 
    \item $V(x)\to \infty $ as $\left\lVert x \right\rVert \to \infty $.  
    \end{enumerate} 
    Then, the zero solution is globally asymptotically stable. If, alternatively, there exists scalars $\alpha ,\beta ,\varepsilon >0$, and $p\geq 1$, such that $V:\mathbb{R} ^{n}\to \mathbb{R} $ satisfies 
    \begin{enumerate}[label=(\roman*)]
    \item $\alpha \left\lVert x \right\rVert ^{p}\leq V(x)\leq \beta \left\lVert x \right\rVert ^{p} $ for $x\in \mathbb{R} ^{n}$,     
    \item $\dot{V}(x)\leq -\varepsilon V(x)$ for $x\in \mathbb{R} ^{n}$,  
    \end{enumerate}   
    then the zero solution is globally exponentially stable.
\end{theorem}
The functions $V(x)$ used in Theorem~\ref{thm3.2} and Theorem~\ref{thm3.3} are called Lyapunov functions, 
which are continuously differentiable and positive definite functions of the states of a given dynamical system.
Theorem~\ref{thm3.2} and Theorem~\ref{thm3.3} show that if the time derivative, or the time rate of change, of the Lyapunov functions due to perturbations in a neighbourhood of the system's equilibrium is always negative or zero, the system's equilibrium point is Lyapunov stable. 
Additionally, if the time rate of change of the Lyapunov function is strictly negative, then the system's equilibrium point is asymptotically stable.

Unlike the Lyapunov direct method, which can provide global stability conclusions for an equilibrium point of a nonlinear dynamical system,
the method we use in this section is the so-called Lyapunov indirect method, 
which uses the linearization of the (possibly nonlinear) system and the spectral information of this linearization. 
Specifically, the local stability of an equilibrium point to small perturbations can be determined by the real parts of the eigenvalues of the Jacobian matrix evaluated at that equilibrium point. 
In particular, if all of the eigenvalues have negative real parts, then any small perturbation will decay exponentially, indicating stability. 
If some of the eigenvalues, say $\lambda $, have a positive real part, any small perturbation along the direction of the eigenvector corresponding to $\lambda $ will be magnified exponentially, 
thus destabilising the system. For a two-dimensional system and the corresponding $2$-by-$2$ Jacobian matrix, all of the two eigenvalues have a negative real part if and only if the trace of the matrix
is negative and the determinant of the matrix is positive, which proves to be a useful criterion for determining the local stability property of equilibrium points of a $2$ dimensional dynamical system.  
\begin{theorem}[\textbf{Lyapunov Indirect Method}]\label{thm3.4} \cite{haddad2008nonlinear}
    Let $x(t)\equiv 0$ be an equilibrium point of (\ref{eq3.1}) where $f:\mathcal{D}\to \mathbb{R} ^{n}$ is continuously differentiable and $\mathcal{D}$ is an open set containing $0$. 
    Furthermore, let $A=\dfrac{\partial f}{\partial x}|_{x=0} $. Then the following statements hold:
    \begin{enumerate}[label=(\roman*)]
    \item If $\mathrm{Re}\,\lambda <0 $ for all $\lambda \in \mathrm{Spec}(A) $, then the zero solution is exponentially stable.      
    \item If there exists $\lambda \in \mathrm{Spec}(A)$ such that $\mathrm{Re}\,\lambda >0$, then the zero solution is unstable.  
    \end{enumerate} 
\end{theorem} 

System (\ref{eq1.3}):
\begin{align*}
    \begin{cases}
\dfrac{dN}{dt} &= N \left( 1-\dfrac{N}{k}  \right) - \dfrac{mNP}{1+N}, \\[6pt]
\dfrac{dP}{dt} &= -cP + \dfrac{mNP}{1+N},     
    \end{cases}
\end{align*} 
is a smooth system in a neighbourhood of the first quadrant $Q=\left\{(N, P): N, P\geq 0 \right\}$ and therefore solutions of initial value problems with nonnegative initial conditions exist and are unique. 
We will show that $Q$ is invariant; i.e., solutions satisfying $N(0),P(0)\geq 0$ satisfy $N(t),P(t)\geq 0$ for all $t$. 
Furthermore, we show that solutions are bounded and thus are defined for all $t\geq 0$.

Note that system (\ref{eq1.3}) has the form:
\begin{align*}
    \begin{cases}
    \dfrac{dN}{dt} &= Nf(N,P), \\[6pt]
    \dfrac{dP}{dt} &= Pg(N,P).
    \end{cases}
\end{align*}
If $(N(t), P(t))$ is a solution of such a system, then $\dfrac{dN}{dt}(t)=N(t)\tilde{f}(t)$ where $\tilde{f}(t)=f(N(t), P(t))$ so by the integrating factor technique for ODEs, we have:
\begin{align*}
    \dfrac{d}{dt}\left(e^{-\int_{0}^{t} \tilde{f}(s)ds}N(t) \right) = 0.
\end{align*}
Consequently,
\begin{align*}
    N(t) = N(0)e^{\int_{0}^{t} \tilde{f}(s)ds}.
\end{align*}
From this, we see that if $N(0)=0$ then $N(t)=0$ for all $t$ and if $N(0)>0$, then $N(t)>0$
for all $t$. Similarly, for $P$. This result not only demonstrates the invariance of $Q$ but also  
proves the invariance of the interior of $Q$ and the boundary of $Q$.

The following boundedness result of solutions can be obtained:
\begin{lemma}[\textbf{Boundedness of Solutions}]\label{lemma3.5}\cite{gunaratne2012rosenzweig}
    There exists $R_{0}>0$ such that for all $R\geq R_{0}$, the right triangle $T(R)$ with sides 
    $N=0$, $P=0$ and $N+P=R$ is positively invariant.      
\end{lemma} 
Since every initial point $(N(0),P(0))\in Q$ satisfies $(N(0),P(0))\in T(R)$ for some $R\geq R_{0}$,
this shows that all solutions starting in $Q$ are bounded for $t\geq 0$ since $(N(t),P(t))\in T(R), t\geq 0$.      

We next turn our attention to our predator-prey model (\ref{eq1.3}), and we hope to find its dynamics.
Let $f(N,P)=N \left( 1-\dfrac{N}{k}  \right) - \dfrac{mNP}{1+N}$ and $g(N,P)=-cP + \dfrac{mNP}{1+N} $. 
To calculate the critical points of (\ref{eq1.3}), we require $f(N,P)=g(N,P)=0$.
When $P=0$, $N=0$ or $N=k$. When $P\neq 0$, $N\neq 0$ and we see in this case 
$N^{*}=\dfrac{c}{m-c},P^{*}=\dfrac{k(m-c)-c}{k(m-c)^{2}}$. 
So, there are three critical points: $K_{1}=(0,0)$, $K_{2}=(k,0)$, 
and $K_{3}=\left(\dfrac{c}{m-c},\dfrac{k(m-c)-c}{k(m-c)^{2}}\right) $. 
The existence of $K_{3}$ requires 
\begin{align}\label{eq3.3}
m>c \quad \text{and} \quad k(m-c)>c.
 \end{align} 
Here, $K_{1}$ means the extinction of two species, $K_{2}$ means the sole existence of the prey, 
and $K_{3}$ means the co-existence of two species. 
We claim that if $m>c$ or $k(m-c)>c$ fails to hold, the predator population will go extinct:
\begin{enumerate}
    \item If $m\leq c$, the predator population will go extinct:
    \begin{align*}
    \dfrac{dP}{dt}=P \left( -c+\dfrac{mN}{1+N}  \right) \leq P \left( m-c \right) \leq  0.    
     \end{align*}    
     So $P(t)\to 0$ as $t\to \infty $.
    \item If $m>c$ and $k(m-c)\leq c$, the prey population will go extinct:
    When $k(m-c)\leq c$, $k\leq \dfrac{c}{m-c} $. We then have \begin{align*}
    N(t)\leq k\leq \dfrac{c}{m-c}, 
     \end{align*}
     which is equivalent to \begin{align*}
     \dfrac{mN}{1+N}\leq c. 
      \end{align*}    
     In this case, \begin{align*}
     \dfrac{dP}{dt}=P \left( -c+\dfrac{mN}{1+N}  \right)\leq 0.  
      \end{align*} 
        So $P(t)\to 0$ as $t\to \infty $.
    \end{enumerate}  

So hereafter, we assume $m>c$ and $k(m-c)>c$ holds. 

The Jacobian matrix of (\ref{eq1.3}) can be calculated as \begin{align*}
J&=\begin{pmatrix}
\dfrac{\partial f}{\partial N}  &\dfrac{\partial f}{\partial P}  \\[8pt]
\dfrac{\partial g}{\partial N}  &\dfrac{\partial g}{\partial P} 
\end{pmatrix} \\ 
&=\begin{pmatrix}
1-\dfrac{2N}{k}-\dfrac{mP}{(1+N)^{2}}   & -\dfrac{mN}{1+N}  \\[8pt]
\dfrac{mP}{(1+N)^{2}}                   & -c+\dfrac{mN}{1+N} 
\end{pmatrix}.
\end{align*} 
The Jacobian at $K_{1}=(0,0)$ is \begin{align*}
J(K_{1})=\begin{pmatrix}
1 & 0 \\
0 & -c
\end{pmatrix}.
 \end{align*}  
and $K_{1}$ is a saddle point.

The Jacobian matrix at $K_{2}=(k,0)$ is \begin{align*}
J(K_{2})=\begin{pmatrix}
-1 & -\dfrac{mk}{1+k}  \\[8pt]
0 & \dfrac{mk}{1+k} - c
\end{pmatrix}.
 \end{align*}     
In view of (\ref{eq3.3}), $K_{2}$ is a saddle point.

The Jacobian matrix at 
$K_{3}=\left(\dfrac{c}{m-c},\dfrac{k(m-c)-c}{k(m-c)^{2}}\right) $ is 
\begin{align*}
J(K_{3})=\begin{pmatrix}
1-\dfrac{2N^{*}}{k}-\dfrac{mP^{*}}{(1+N^{*})^{2}} &-c \\[8pt]
\dfrac{mP^{*}}{(1+N^{*})^{2}}                     &0
\end{pmatrix}.
 \end{align*}  
 It has a positive determinant, so local stability is determined by the trace, which, by simple calculations,
can be written as 
\begin{align*}
\left( 1+N^{*} \right)\mathrm{Trace}(J(K_{3})) = \left( 1-\dfrac{2N^{*}}{k}  \right)\left( 1+N^{*} \right)-\left( 1-\dfrac{N^{*}}{k}\right)
                                             = \dfrac{2N^{*}}{k}\left( \dfrac{k-1}{2}-N^{*}  \right)     
\end{align*} 
We then have the following result:
\begin{lemma}[\textbf{Stability of Positive Equilibrium}]\label{lemma3.6} \cite{gunaratne2012rosenzweig}
    Coexistence equilibrium $K_{3}=(N^{*},P^{*})$ is a sink if $N^{*}>\dfrac{k-1}{2}$ and a source if $N^{*}<\dfrac{k-1}{2}$.   
\end{lemma} 
We can make the above conclusion more clear by the value of $N^{*}=\dfrac{c}{m-c}$. 
$K_{3}$ is a sink if $k<\dfrac{m+c}{m-c} $ and a source if $k>\dfrac{m+c}{m-c}$.
So a Hopf bifurcation occurs at $k=\dfrac{m+c}{m-c}$, and the stability of $K_{3}$ changes from a sink to a source as the real part of the eigenvalues of the Jacobian changes sign from negative to positive. 
We also have the following result:
\begin{theorem}[\textbf{Stable Limit Cycle}]\label{thm3.7} \cite{gunaratne2012rosenzweig} 
    Assume that $N^{*}<\dfrac{k-1}{2}$. Then there exists a periodic solution of (\ref{eq1.3}). 
    Every solution starting in the interior of $Q$, except for the positive equilibrium, has a periodic orbit as its limit set.  
\end{theorem} 
\begin{theorem}[\textbf{Stable Positive Equilibrium}]\label{thm3.8} \cite{gunaratne2012rosenzweig}
    Assume that $N^{*}>\dfrac{k-1}{2}$. Then there are no periodic orbits of (\ref{eq1.3}).
    Every solution starting in the interior of $Q$ converges to the positive equilibrium.  
\end{theorem} 
In \cite{cheng1981uniqueness}, the author showed that the periodic orbit in Theorem~\ref{thm3.7} is unique, and therefore it is the omega limit set of every solution starting in the interior of $Q$.
The phase portraits with vector fields are depicted in Figure~\ref{fig:fig} 
for the case $N^{*}<\dfrac{k-1}{2}$ ($m=3,c=1,k=3$) and $N^{*}>\dfrac{k-1}{2}$ ($m=3,c=1,k=1.5$), respectively.

\begin{figure}[H]
    \centering
    \hspace*{-1.6cm} \includegraphics[width=1.1\textwidth]{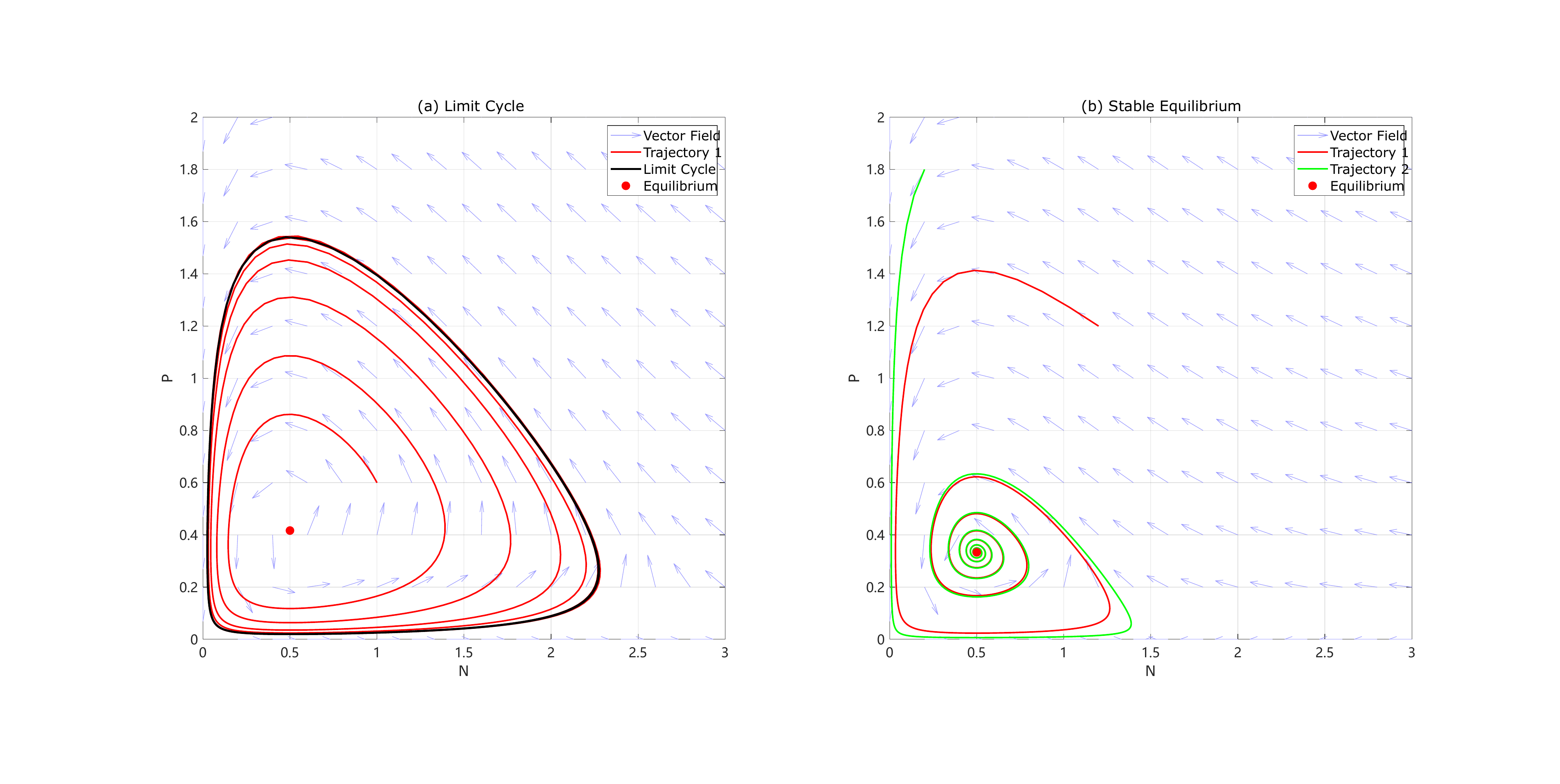}
    \caption{Figure (a) is the phase portrait of the system (\ref{eq1.3}) with $m=3,c=1,k=3$, 
    and Figure (b) is the phase portrait of the system (\ref{eq1.3}) with $m=3,c=1,k=1.5$.
    The short arrows are approximations of the vector field of the system (\ref{eq1.3}). 
    For the stable limit cycle, another trajectory with this cycle as its omega limit set is shown in Figure (a); 
    while two additional trajectories with the positive equilibrium as their omega limit set are shown in Figure (b).}
    \label{fig:fig}
\end{figure}

\section{Stochastic Properties}
\label{sec:stochastic}

We first fix a probability space $(\Omega ,\mathcal{F} ,P)$ and a $d$-dimensional Brownian motion $W = \left\{W_{t},\mathcal{F}_{t}^{W};0\leq t<\infty \right\}$ on it, where $\mathcal{F}_{t}^{W}=\sigma \left(W_{s};0\leq s\leq t \right)$. 
When the initial condition $X_{0}=\xi\in \mathbb{R}^{d}$ is fixed and independent of $\mathcal{F} _{\infty }^{W}=\sigma \left( \bigcup\limits_{t\geq 0}\,\mathcal{F} _{t}^{W} \right)$, we construct an 
appropriate filtration as follows: Let \begin{align*}
\mathcal{G}_{t}=\sigma (\xi)\vee \mathcal{F} _{t}^{W}=\sigma \left(\xi,W_{s};0\leq s\leq t\right) \mbox{ for } t\in [0,\infty ),
\end{align*}
and $\mathcal{N}$ be the collection of $P$-null sets in $\mathcal{G}_{\infty}$. The augmentation filtration
\begin{align*}
\mathcal{F} _{t}&=\sigma \left( \mathcal{G}_{t}\cup \mathcal{N} \right), 0\leq t< \infty; \\[6pt]
\mathcal{F}_{\infty}&=\sigma \left( \bigcup\limits_{t\geq 0}\,\mathcal{F} _{t} \right)
\end{align*}
satisfies the usual hypothesis (i.e., it is right-continuous, and $\mathcal{F} _{0}$ contains all the $P$-null sets in $\mathcal{F}$) 
and $W$ is a Brownian motion w.r.t. $\left\{ \mathcal{F} _{t} \right\} $\cite{karatzas1991brownian}. 

We seek a strong solution to (\ref{eq1.4}). 
Let \begin{align}\label{eq4.1} 
\begin{aligned}
\mu(N,P)&=\begin{pmatrix}
N \left( 1-\dfrac{N}{k}  \right) - \dfrac{mNP}{1+N} \\[8pt] 
-cP + \dfrac{mNP}{1+N}
\end{pmatrix},  \\  
\sigma (N,P) &=\begin{pmatrix}
\left[N \left( 1+\dfrac{N}{k}  \right) + \dfrac{mNP}{1+N} \right]^{1/2} & 0 \\[8pt] 
0 & \left[cP + \dfrac{mNP}{1+N} \right]^{1/2}
\end{pmatrix}.
\end{aligned}
\end{align}
By the classical existence and uniqueness result for the strong solution of SDE, if the coefficients $\mu$ and $\sigma$ are globally Lipschitz continuous 
and satisfy the linear growth condition, then (\ref{eq1.3}) admits a unique strong solution. However, global Lipschitz continuity and the linear growth condition are sometimes too restrictive to ask for. 
When the Lipschitz continuity or linear growth condition is violated, we can still seek the existence and uniqueness of strong solutions. 

Here comes one result in this direction:
\begin{lemma}[\textbf{Existence and Uniqueness of Strong Solution, Domain Invariance}]\label{lem4.1} \cite{abundo1991stochastic,khasminskii2012stochastic} 
    Let $\left\{ K_{n} \right\} $ be an increasing sequence of open sets whose closures are contained in $U$ 
    and such that $\bigcup\limits_{n}\,K_{n}=U$, 
    or $\left\{ K_{n} \right\} $ be an increasing sequence of compact sets such that $\bigcup\limits_{n}\,K_{n}=U$,
     where $U$ is an open set of $\mathbb{R} ^{d}$. 
    Let us consider the system: \begin{align}\label{eq4.2}
    dX_{t}=\mu(t,x)dt+\sigma (t,x)dW_{t},
     \end{align}   
     where $W$ is a $d$-dimensional Brownian motion with independent components and the coefficients $\mu(t,x), \sigma (t,x)$ are Lipschitz continuous and bounded in each set $\mathbb{R} ^{+}\times K_{n}$;
     let $V:\mathbb{R} ^{+}\times U\to \mathbb{R} $ be a nonnegative function that is $C^{1}$ in $t\in \mathbb{R} ^{+}$ and $C^{2}$ in $x\in U$ such that \begin{enumerate}[label=(\roman*)]
     \item $\exists \,\, \alpha >0$ such that $LV\leq \alpha V$;     
     \item $\underset{t>0,\,x\in U \setminus K_{n}}{\inf}\, V(t,x)\to \infty $ for $n\to \infty $,  
     \end{enumerate} 
     where $L$ is the infinitesimal generator of (\ref{eq4.2}): 
     \begin{align*}
    Lf=\sum\limits_{k=1}^{d}\mu_{i}(t,x)\dfrac{\partial f}{\partial x_{i}}+\dfrac{1}{2}\sum\limits_{i,j=1}^{d}a_{ij}(t,x)\dfrac{\partial f^{2}}{\partial x_{i}\partial x_{j}}    
\end{align*} 
    for $f\in C^{2}(\mathbb{R} ^{+}\times \mathbb{R} ^{+})$, where $a(t,x)=\sigma (t,x)\sigma (t,x)^{T}$. 
     Assume for (\ref{eq3.2}) the initial condition $X_{0}=\xi$ satisfies $P(\xi \in U)=1$, 
     then there exists a unique strong solution $X_{t}$ of (\ref{eq3.2}) up to distinguishability.
     Moreover, the solution satisfies the relation: \begin{align*}
      P (X_{t}\in U)=1  \mbox{ for all } t\geq 0.
       \end{align*} 
\end{lemma}

Before proceeding to the main theorem, we first prove an extension result for continuous functions defined 
on a closed subset of a manifold.
\begin{lemma}[Extension Lemma for Continuous Functions] \label{lemma4.2}
    Suppose $M$ is a smooth manifold wiith or without boundary, $A\subseteq M$ is a closed subset, 
    and $f:A\to \mathbb{R} ^{k}$ is a continuous function. For any open subset $U$ containing $A$,
    there exists a continuous function $\tilde{f}:M\to \mathbb{R} ^{k}$ 
    such that $\tilde{f}|_{A}=f$ and $\mathrm{supp}\tilde{f}\subseteq U $.        
\end{lemma} 

The proof is the same as the proof for the extension lemma for smooth functions in \cite{lee2012smooth}. 
Since in the proof of that result, the only technique used is the smooth partition of unity, and when we use the partition of unity to glue local objects, we get a global function with the same regularity as the function defined on the closed subset.

\begin{theorem}[\textbf{Existence and Uniqueness of Strong Solution, Domain Invariance}]\label{thm4.3}
    Let $\left\{ K_{n} \right\} $ be an increasing sequence of closed subsets of $U$ such that $\bigcup\limits_{n}\,K_{n}=U$, where $U$ is an open set of $\mathbb{R} ^{d}$.
    Additionally, for each $K_{n}$, there exists an open and precompact subset $\tilde{K}_{n}$ of $U$ containing $K_{n}$. 
    Let us consider the system: \begin{align}\label{eq4.3}
    dX_{t}=\mu(x)dt+\sigma (x)dW_{t},
     \end{align}   
    where $W$ is a $d$-dimensional Brownian motion with independent components 
    and the coefficients $\mu(x), \sigma (x)$ are Lipschitz continuous and bounded in each set $K_{n}$;
    let $V:\mathbb{R} ^{+}\times U\to \mathbb{R} $ be a nonnegative function that is $C^{1}$ in $t\in \mathbb{R} ^{+}$ and $C^{2}$ in $x\in U$ such that \begin{enumerate}[label=(\roman*)]
     \item $\exists \,\, \alpha >0$ such that $LV\leq \alpha V$;     
     \item $\underset{t>0,\,x\in U \setminus K_{n}}{\inf}\, V(t,x)\to \infty $ for $n\to \infty $,  
     \end{enumerate} 
     where $L$ is the infinitesimal generator of (\ref{eq4.3}): 
     \begin{align*}
    Lf=\sum\limits_{k=1}^{d}\mu_{i}(x)\dfrac{\partial f}{\partial x_{i}}+\dfrac{1}{2}\sum\limits_{i,j=1}^{d}a_{ij}(x)\dfrac{\partial f^{2}}{\partial x_{i}\partial x_{j}}    
\end{align*} 
    for $f\in C^{2}(\mathbb{R} ^{+}\times \mathbb{R} ^{+})$, where $a(x)=\sigma (x)\sigma (x)^{T}$. 
     Assume for (\ref{eq4.3}) the initial condition $X_{0}=\xi$ satisfies $P(\xi \in U)=1$, 
     then there exists a unique strong solution $X_{t}$ of (\ref{eq4.3}) up to distinguishability.
     Moreover, the solution satisfies the relation: \begin{align*}
      P (X_{t}\in U)=1  \mbox{ for all } t\geq 0.
       \end{align*} 
\end{theorem} 

\begin{proof}
We use the localization technique. The Lipschitz condition of $\mu$ and $\sigma$ are valid in every $K_{n}$, 
and we can construct a sequence of truncated functions $\mu_{n}$ and $\sigma _{n}$ such that for $x\in K_{n}$ \begin{align*}
\mu_{n}(x)=\mu(x),\quad \sigma_{n}(x)=\sigma (x),
 \end{align*}      
and by Lemma~\ref{lemma4.2},
for each $n$, there exists continuous functions $\tilde{\mu}_{n}$ and $\tilde{\sigma }_{n}$ defined on $U$ such that 
\begin{align*}
&\tilde{\mu }_{n}|_{K_{n}}=\mu _{n}, \quad \mathrm{supp}\,\tilde{\mu }_{n}\subseteq \tilde{K}_{n}, \\[6pt] 
&\tilde{\sigma }_{n}|_{K_{n}}=\sigma _{n},\quad \mathrm{supp}\,\tilde{\sigma }_{n}\subseteq \tilde{K}_{n}. 
 \end{align*}  
This means that we can extend $\mu _{n}$ and $\sigma _{n}$ to continuous functions on $U$ whose supports are in $\tilde{K}_{n}$, 
and we use this extension as the definition for $\mu _{n}$ and $\sigma _{n}$.  
The initial condition is truncated as follows: $\xi_{n}(\omega )=\xi(\omega )$ if $\left\lvert \xi (\omega ) \right\rvert \leq n$ and equals to $0$ otherwise. 
Since for each $n$, $\mathrm{supp}\,\tilde{\mu }_{n}, \mathrm{supp}\,\tilde{\sigma }_{n}\subseteq \tilde{K}_{n}$
and $\tilde{K}_{n}$ is precompact, it is easy to see each $\mu _{n}$ and $\sigma _{n}$ are globally Lipschitz continuous on $U$ and therefore they satisfy the linear growth condition. 
Thus, by the classical existence and uniqueness for solutions of SDEs, there exists a unique strong solution $X^{n}$ that satisfies 
\begin{align}\label{eq4.4}
X_{t}^{n}=\xi _{n}+\int_{0}^{t}\mu _{n}(X_{s}^{n})\,ds+
\int_{0}^{t}\sigma _{n}(X_{s}^{n})\,dW_{s},\quad t\geq 0.
 \end{align}  
Let $\tau _{n}=\inf\,\left\{ t\geq 0: X_{t}^{n}\not\in K_{n} \right\} $. 
It can be shown \cite{friedman1975stochastic} that for $n ^{\prime} \geq n$, 
we have $X_{t}^{n ^{\prime} }(\omega )=X_{t}^{n}(\omega )$ a.s. on $t\leq \tau _{n}(\omega )$. 
Then it is natural to define $X_{t}:=X_{t}^{n}$ for $t\in [0,\tau _{n}]$ and by letting $n\to \infty $ in (\ref{eq4.4}) 
we obtain a local solution up to the time $\tau _{\infty }:=\underset{n\to \infty }{\lim}\,\tau _{n}$. 

We next show $P(\tau _{n}=\infty )=1$, and this finishes the proof of the existence part. Let \begin{align*}
W(t,x)=V(t,x)e^{-\alpha t},
 \end{align*} 
then by condition (\romannumeral 1 ) in Theorem~\ref{thm4.3}, $\left( L+\partial _{t} \right)W\leq 0$. 
Hence, by It\^o formula and then taking the expectation, we obtain \begin{align*}
    &E \left[ V(\tau _{n}\wedge t,X_{\tau _{n}\wedge t}^{n})e^{-\alpha (\tau _{n}\wedge t)} \right]-E\left[ V(0,X_{0}^{n}) \right] \\ 
    &=E \int_{0}^{\tau _{n}\wedge t}\left( L+\partial _{t} \right)W(s,X_{s})\,du \leq 0.
 \end{align*} 
This, together with $\tau _{n}\wedge t\leq t$ and the definition of $X$ imply that \begin{align*}
EV(\tau _{n}\wedge t,X_{\tau _{n}\wedge t})\leq e^{\alpha t }EV(0,X_{0}). 
 \end{align*} 
 Then we have: \begin{align*}
     EV(\tau _{n}\wedge t,X_{\tau _{n}\wedge t})&=\int_{\tau_{n}\leq t} V(\tau_{n},X_{\tau_{n}}) d\omega +          
                                                  \int_{\tau_{n}>t}V(t,X_{t})d\omega \\
                                                &\geq \int_{\tau_{n}\leq t} V(\tau_{n},X_{\tau_{n}}) d\omega \\ 
                                                &\geq P(\tau _{n}\leq t)\underset{\mathbb{R} ^{+}\times\left(U\setminus K_{n}\right)}{\inf}\,V(s,x),
 \end{align*}
 where the last inequality follows from the following fact ($\mathrm{bd} $ means the topological boundary):
 $X_{\tau _{n}} = X_{\tau _{n}}^{n}\in \mathrm{bd}(K_{n}) $, and the continuity of $V$ implies that 
 \begin{align*}
     \underset{\mathbb{R} ^{+}\times\left(U\setminus K_{n}\right)}{\inf}\,V(s,x) \leq \underset{\mathbb{R} ^{+}\times \mathrm{bd}(K_{n})}{\inf}\,V(s,x).
 \end{align*} 
 We derive the estimate \begin{align*}
 P(\tau _{n}\leq t)\leq \dfrac{e^{\alpha t}EV(0,X_{0})}{\underset{\mathbb{R} ^{+}\times\left(\tilde{U}\setminus \tilde{K}_{n}\right)}{\inf}\,V(s,x)}, 
  \end{align*}  
  and, by the second condition (\romannumeral 2 ) of $V$, $P(\tau _{\infty }=\infty )=1$ follows by letting $n\to \infty $ in the above estimate. This solution is unique up to indistinguishability. 
  Indeed, it follows from the definition of $X_{t}$ and the uniqueness of the solution to (\ref{eq4.4}) that for every pair of solutions $X_{t}$ and $Y_{t}$ \begin{align*}
  P \left\{ \underset{0<t<\tau _{n}}{\sup}\,\left\lVert X_{t}-Y_{t} \right\rVert >0 \right\}=0. 
   \end{align*} 
   The desired result follows by letting $n\to \infty $ and using $P(\tau _{\infty }=\infty )=1$.    
\end{proof}
\begin{remark}\label{rmk4.4}
\begin{enumerate}
    \item  When proving the existence and uniqueness of solutions, many typical Lyapunov functions, such as $V(x) = \sum\limits_{i=1}^{n} \left( x_{i}-1-\log(x_{i}) \right)$ or $V(x)=\left( 1+\left\lVert x \right\rVert _{2}^{2} \right)^{\alpha } $ for some $\alpha >0$, does not satisfy (\romannumeral 2) if ${K_{n}}$ is as required in Lemma~\ref{lem4.1}. Theorem~\ref{thm4.3} allows us to use these Lyapunov functions.
    \item  The additional requirements about the existence of an open and precompact subset $\tilde{K}_{n}$ containing $K_{n}$ can be easily satisfied in most of the application scenarios.
    \end{enumerate} 
\end{remark} 

Using the above result, we can prove the following:
\begin{theorem}[\textbf{Existence and Uniqueness of Strong Solution, Domain Invariance}]\label{thm4.5}

    For the initial condition starting from $D=\left\{ (N,P)\in \mathbb{R} ^{2}:N> 0,P> 0  \right\} $, 
    (\ref{eq1.4}) admits a unique strong solution.
    Moreover, we have the invariance of $D$: 
    \begin{align*} 
    P((N(0),P(0))=\xi \in D)=1 \mbox{implies} P((N(t),P(t))\in D)=1 \mbox{ for } t\geq 0.
    \end{align*}
\end{theorem} 
\begin{proof}  
    Let 
    \begin{align*}
    K_{n}&=\left\{ (N,P)\in \mathbb{R} ^{2}: 0< N\leq n, 0< P\leq n \right\}, \\ 
    \tilde{K}_{n}&= \left\{ (N,P)\in \mathbb{R} ^{2}: 0< N <n+1, 0< P<n+1 \right\},
     \end{align*} 
    and \begin{align*}
    V(N,P)=\left( 1+N^{2}+P^{2} \right)^{\alpha }, 
     \end{align*} 
    where $\alpha >0$ is some positive number to be determined later. 
    We denote $U(N,P)= 1+N^{2}+P^{2}$ so that $V= U^{\alpha }$. 
    $\mu$ and $\sigma $ have continuous first-order partial derivatives, satisfying the Lipschitz and linear growth conditions in each $K_{n}$.
    It can be readily checked that \begin{align*}
    LV
    &=\underbrace{\left( 2 \alpha N U(N,P)^{\alpha -1}\right)   \left( N \left( 1-\dfrac{N}{k} -\dfrac{mP}{1+N}   \right) \right)}_{I_{1}} \\
    &+\underbrace{\left( 2 \alpha P U(N,P)^{\alpha -1}\right) \left( P \left( -c+\dfrac{mN}{1+N}  \right)\right)}_{I_{2}}\\ 
       &+\underbrace{\alpha U(N,P)^{\alpha -1}\left( N \left( 1+\dfrac{N}{k}+\dfrac{mP}{1+N}   \right) \right)}_{I_{3}} \\
       &+\underbrace{2 \alpha (\alpha -1)N^{2} U(N,P)^{\alpha -2}\left( N \left( 1+\dfrac{N}{k}+\dfrac{mP}{1+N}   \right) \right)}_{I_{4}} \\ 
       &+\underbrace{\alpha U(N,P)^{\alpha -1}\left(P \left( c+\dfrac{mN}{1+N}  \right)\right)}_{I_{5}} \\
       &+\underbrace{2 \alpha (\alpha -1)P^{2} U(N,P)^{\alpha -2}\left(P \left( c+\dfrac{mN}{1+N}  \right)\right)}_{I_{6}}. 
     \end{align*}       
For $I_{1}$:
\begin{align*}
I_{1}\leq 2 \alpha N^{2} U(N,P)^{\alpha -1}\leq 2\alpha V(N,P).
 \end{align*}  
 For $I_{2}$:
 \begin{align*}
 I_{2}\leq 2\alpha m P^{2}U(N,P)^{\alpha -1}\leq 2\alpha m V(N,P).
  \end{align*}  
  For $I_{3}$:
   \begin{align*}
   I_{3} = \underbrace{\alpha  U(N,P)^{\alpha -1}N}_{I_{3,1}}
   + \underbrace{\alpha U(N,P)^{\alpha -1}\dfrac{N^{2}}{k} }_{I_{3,2}}
   +\underbrace{\alpha U(N,P)^{\alpha -1}\dfrac{mP}{1+N} }_{I_{3,3}}
    \end{align*} 
    For $I_{3,1}$, by Young inequality, we have $N\leq \dfrac{1+N^{2}}{2} $
    and 
    \begin{align*}
    I_{3,1}\leq \dfrac{\alpha }{2}U(N,P)^{\alpha -1}(1+N^{2}) \leq \dfrac{\alpha }{2}V(N,P). 
     \end{align*}    
    For $I_{3,2}$:
    \begin{align*}
    I_{3,2}\leq \dfrac{\alpha }{k}V(N,P). 
     \end{align*}  
     For $I_{3,3}$:
     \begin{align*}
     I_{3,3}\leq \alpha U(N,P)^{\alpha -1}P\leq \dfrac{\alpha }{2} U(N,P)^{\alpha -1}(1+P^{2}) \leq \dfrac{\alpha }{2}V(N,P).
      \end{align*} 
      For $I_{4}$:
      \begin{align*}
      I_{4} &= \underbrace{2 \alpha (\alpha -1)N^{2} U(N,P)^{\alpha -2}N}_{I_{4,1}}
      +\underbrace{2 \alpha (\alpha -1)N^{2} U(N,P)^{\alpha -2}\dfrac{N^{2}}{k}}_{I_{4,2}} \\
      &+\underbrace{2 \alpha (\alpha -1)N^{2} U(N,P)^{\alpha -2}\dfrac{mP}{1+N}}_{I_{4,3}}.
       \end{align*}  
    We have \begin{align*}
    I_{4,1}&\leq \alpha \left( \alpha -1 \right)V(N,P),\\
    I_{4,2}&\leq \dfrac{2\alpha (\alpha -1)}{k}V(N,P), \\ 
    I_{4,3}&\leq \alpha (\alpha -1)m V(N,P). 
     \end{align*} 
     For $I_{5}$:
     \begin{align*}
     I_{5} = \underbrace{\alpha U(N,P)^{\alpha -1}cP}_{I_{5,1}}
     +\underbrace{\alpha U(N,P)^{\alpha -1}P\dfrac{mN}{1+N}}_{I_{5,2}}.
      \end{align*}  
      We have \begin{align*}
      I_{5,1}&\leq \dfrac{\alpha c}{2}V(N,P), \\ 
      I_{5,2}&\leq \dfrac{\alpha m}{2}V(N,P). 
       \end{align*} 
       For $I_{6}$:
        \begin{align*}
        I_{6}=\underbrace{2 \alpha (\alpha -1)P^{2} U(N,P)^{\alpha -2}cP}_{I_{6,1}}
        +\underbrace{2 \alpha (\alpha -1)P^{2} U(N,P)^{\alpha -2}P \dfrac{mN}{1+N} }_{I_{6,2}}.
         \end{align*} 
         We have \begin{align*}
         I_{6,1}&\leq \alpha (\alpha -1)c V(N,P), \\
         I_{6,2}&\leq \alpha (\alpha -1)m V(N,P).
          \end{align*} 
    Choose $\alpha >2$, this lower bound for $\alpha $ suffices in the above derivations, and let \begin{align*} 
        C_{1}=3\alpha+ \dfrac{5\alpha m}{2}+ \dfrac{\alpha c}{2}+\dfrac{\alpha }{k}+\alpha (\alpha -1)\left( 1+\dfrac{2}{k}+2m+c  \right).    
    \end{align*}
    Then, \begin{align*}
     LV\leq C_{1}V.
      \end{align*} 
    It is easy to see that $\underset{t>0,\,(N,P)\in D \setminus K_{n}}{\inf}\, V(N,P) = \left( n^{2}+1 \right)^{\alpha } \to \infty $ as $n\to \infty $. 
    Theorem~\ref{thm4.3} implies the existence and uniqueness of the strong solution $X_{t}$ to (\ref{eq1.3}) up to indistinguishability and the invariance of $D$.    
\end{proof}

\begin{remark}[\textbf{Significance of the Domain Invariance of $D$}]\label{rmk4.6}

The importance of the invariance of $D=\left\{ (N,P):N,P>0 \right\} $ can be emphasized in the following two aspects:
    \begin{enumerate}[label=(\roman*)]
    \item The invariance of the domain $D$ is obvious in the biological settings since $N$ and $P$ represent the density of the prey and the predator, which can only be nonnegative, and the nonnegativity is sometimes implicitly assumed and incorporated in the model by some authors. Here we give a rigorous proof of the nonnegativity.
    \item The invariance of $D$ implies that the behavior of $\mu$ and $\sigma $ outside of $D$ will not influence the behavior of the solution, so we can extend many results that require the behaviors of $\mu $ and $\sigma $ on the whole domain $\mathbb{R} ^{+}\times \mathbb{R} ^{d}$ to the result that only requires the behavior of $\mu $ and $\sigma $ on the domain $\mathbb{R} ^{+}\times D$ of our interest.
    \end{enumerate}
\end{remark}     

There are two statuses of the predator and prey populations: extinction (the density of the population becomes zero) and persistence (the density of the population remains positive).  
Since $0$ is not in $D$, this invariance excludes the possibility of population extinction. 

We know external (environmental) randomness results in destabilizing effects and leads to the extinction of the population \cite{feldman1975population,may2019stability}. 
Our predator and prey system only incorporates internal randomness in the population, and the nonzero initial data can guarantee the global-in-time persistence of the population (the population persists for all positive time). Therefore, the internal randomness does not have the same destabilizing effects and leads to population extinction as the external randomness.

We state some results about the $p$-th moment estimates of the solution and then apply them to our system (\ref{eq3.2}). 
\begin{lemma}[\textbf{$L^{p}$ Estimate}]\label{lem4.7} \cite{mao2007stochastic}
    \begin{enumerate}[label= (\roman*)]
    \item Let $0<p<2$ and $X_{t_{0}}\in L^{2}(\Omega ;\mathbb{R} ^{d})$. Assume that there exists a constant $\alpha >0$ such that for all $(t,x)\in [t_{0},\infty )\times \mathbb{R} ^{d}$, 
    \begin{align}\label{eq4.5}
        x^{T}\mu(t,x)+\dfrac{1}{2}\left\lVert \sigma (t,x) \right\rVert ^{2}\leq \alpha (1+ \left\lVert x \right\rVert ^{2}). 
     \end{align}    
     Then \begin{align*}
        E \left\lVert X_{t} \right\rVert ^{p}\leq \left( 1+E \left\lVert X_{t_{0}}  \right\rVert ^{2} \right)^{p/2}e^{p \alpha (t-t_{0})} \mbox{ for } t\geq t_{0}. 
      \end{align*} 
    \item Let $p\geq 2$ and $X_{t_{0}}\in L^{p}(\Omega ;\mathbb{R} ^{d})$. Assume that there exists a constant $\alpha >0$ such that for all $(t,x)\in [t_{0},\infty )\times \mathbb{R} ^{d}$, 
    \begin{align*}
    x^{T}\mu(t,x)+\dfrac{p-1}{2}\left\lVert \sigma (t,x) \right\rVert ^{2}\leq \alpha (1+ \left\lVert x \right\rVert ^{2}). 
    \end{align*}   
    Then \begin{align*}
    E \left\lVert X_{t} \right\rVert ^{p}\leq 2^{(p-2)/2} \left( 1+E \left\lVert X_{t_{0}}  \right\rVert ^{p} \right)e^{p \alpha (t-t_{0})} \mbox{ for } t\geq t_{0}. 
     \end{align*}  
    \end{enumerate}
\end{lemma} 
The condition (\ref{eq4.5}) is called the monotonicity condition, and this condition, together with the local Lipschitz continuous condition 
can imply the existence and uniqueness of the strong solution \cite{mao2007stochastic}. 
Lemma~\ref{lem4.7} give the $L^{p}$ estimate of the solution when $p< 2$ and $p\geq 2$.  

\begin{proposition}[\textbf{$L^{p}$ Estimate}]\label{prop4.8}
    For the initial condition starting from $D=\left\{ (N,P)\in \mathbb{R} ^{2}:N> 0,P> 0  \right\} $ 
    and belong to $L^{2}(\Omega ;\mathbb{R} ^{d})$, 
    the unique strong solution to (\ref{eq1.4}) have the following $p$-th moment estimates: 
    \begin{align*}
        &E \left\lVert X_{t} \right\rVert ^{p}\leq \left( 1+E \left\lVert \xi  \right\rVert ^{2} \right)^{p/2}e^{p C_{1}t} \mbox{ for } 0<p<2, t\geq 0, \\ 
        &E \left\lVert X_{t} \right\rVert ^{p}\leq 2^{(p-2)/2} \left( 1+E \left\lVert \xi  \right\rVert ^{p} \right)e^{p C_{2}t} \mbox{ for } p\geq 2, t\geq 0.
     \end{align*} 
     where $C_{1}=\dfrac{5+6m+c}{4}+\dfrac{1}{2k}  $ and $C_{2}=1+m+\dfrac{p-1}{4}\left( 1+2m+c \right)+\dfrac{p-1}{2k}$. 
\end{proposition}
\begin{proof}
    We only derive the first estimate involving $C_{1}$, and the derivation of the other is similar. \begin{align*}
    &x^{T}\mu(x)+\dfrac{1}{2}\left\lVert \sigma (x) \right\rVert ^{2}  \\ 
    =& x^{2}\left( 1-\dfrac{x}{k}  \right)-\dfrac{mx^{2}y}{1+x}-cy^{2}+\dfrac{mxy^{2}}{1+x} \\ 
     &+ \dfrac{1}{2}x \left( 1+\dfrac{x}{k}+\dfrac{my}{1+x}   \right)+\dfrac{1}{2}y \left( c+\dfrac{mx}{1+x}  \right)  \\ 
     &\leq  x^{2}+my^{2}+\dfrac{x}{2}+\dfrac{x^{2}}{2k}+\dfrac{my}{2}+\dfrac{cy}{2}+\dfrac{my}{2} \\ 
     &\leq \left(\dfrac{5+6m+c}{4}+\dfrac{1}{2k}\right) \left( 1+\left\lVert x \right\rVert _{2}^{2} \right)      
     \end{align*}  
\end{proof} 
Proposition~\ref{prop4.8} shows that the $p$-th moment of the solution will grow at most exponentially with exponent $pC_{1}$ when $p< 2$ and with exponent $pC_{2}$ when $p\geq 2$. In other words, \begin{align*}
\underset{t\to \infty }{\limsup}\, \dfrac{1}{t}\log \left( E \left\lVert X_{t} \right\rVert ^{p} \right)\leq pC_{i}.
 \end{align*}     
The quantity on the left-hand side is called the $p$-th moment Lyapunov exponent and should not exceed $pC_{i}$. We proceed to give an estimate of the Lyapunov exponent that is defined to be \begin{align*}
\underset{t\to \infty }{\limsup}\,\dfrac{1}{t}\log \left\lVert X_{t} \right\rVert . 
 \end{align*}   
\begin{theorem}[\textbf{Asymptotic Estimamte}]\label{thm4.9} \cite{mao2007stochastic}

   The Lyapunov exponent of the unique strong solution to (\ref{eq1.4}) should not exceed $C_{1}=\dfrac{5+6m+c}{4}+\dfrac{1}{2k} $, that is \begin{align*}
   \underset{t\to \infty }{\limsup}\, \dfrac{1}{t}\log \left\lVert X_{t} \right\rVert \leq C_{1} \quad a.s. 
    \end{align*} 
\end{theorem}
The asymptotic estimate presented above indicates that the population will exhibit at most exponential growth with the exponent given by $C_{1}=\dfrac{5+6m+c}{4}+\dfrac{1}{2k} $. 
This conclusion aligns with the deterministic ecological model analysis, which suggests that populations typically experience exponential growth. However, in a more realistic environment, growth is constrained by limited resources, leading to a transition from exponential growth to a more stable growth limited by the carrying capacity of the ecosystem.
To be more precise, in an ideal environment with sufficient resources (such as food), a population will experience exponential growth. Conversely, in a more realistic ecosystem,
the population will undergo limited growth, with the maximum population number not exceeding the environmental carrying capacity due to resource competition. 

In this stochastic version of the Rosenzweig-MacArthur predator-prey model, we incorporate internal randomness arising from fluctuations in growth, reproduction, and deaths of individuals in the population. The asymptotic estimate reveals that this stochastic model, which includes internal population fluctuations, does not significantly deviate from the corresponding deterministic models. Both systems predict that population growth will be at most exponential.

\section{Numerical Results}
\label{sec:numerical}

In this section, we deal with the numerical solution of (\ref{eq1.3}) 
with initial condition $X_{0}=(N(0),P(0))=(N_{0},P_{0})$. 
For the general SDE of the form \begin{align}\label{eq5.1}
dX_{t}=\mu (x,t)dt+\sigma (x,t)dW_{t}, 
 \end{align}   
one simple time discrete approximation of (\ref{eq5.1}) is the Euler approximation or the Euler-Maruyama approximation. 
We assume (\ref{eq5.1}) has a unique strong solution on $[0,\infty )$, and in numerical simulations we are primarily interested in the behavior of the solution up to some terminal time $T>0$. 
For simplicity, we adopt the equidistant time discretization $0=\tau _{0}<\tau _{1}<\ldots <\tau _{i}<\ldots <\tau _{M}=T$, with $\Delta =\dfrac{T}{M}$ and $\tau _{i}=i\Delta=\dfrac{iT}{M}$. 
Then an Euler approximation under this equidistant time discretization is a continuous time stochastic process $Y=\left\{ Y_{t}:0\leq t\leq T \right\} $ satisfying the iterative scheme \begin{align*}
Y_{n+1}=Y_{n}+\mu (\tau _{n},Y_{n})\Delta +\sigma (\tau _{n},Y_{n})\left( W_{\tau _{n+1}}-W_{\tau _{n}} \right) ,
 \end{align*}  
for $n=0,1,\ldots ,N-1$ with initial value \begin{align*}
Y_{0}=X_{0},
 \end{align*}  
where we have written \begin{align*}
Y_{n}=Y_{\tau _{n}}
 \end{align*} 
for the value of the approximation at the discretization time $\tau _{n}$ \cite{kloeden1992stochastic}. We also denote the increments of the Brownian motion by $\Delta W_{n}:=W_{\tau _{n+1}}-W_{\tau _{n}}$,
for $n=0,1,\ldots ,M-1$. These increments are independent Gaussian random variables with mean \begin{align*}
\mathrm{E} \left( \Delta W_{n} \right)=0, 
 \end{align*}  
 and variance \begin{align*}
 \mathrm{Var} \left( \Delta W_{n} \right)=\Delta =\dfrac{T}{M}.  
  \end{align*} 
Now, the Euler scheme for (\ref{eq1.4}) is as follows: \begin{equation}
    \begin{split}
&N_{n+1}=N_{n}+N_{n}\left( 1-\dfrac{N_{n}}{k} -\dfrac{mP_{n}}{1+N_{n}}\right) \Delta+ \left[ N_{n}\left( 1+\dfrac{N_{n}}{k} +\dfrac{mP_{n}}{1+N_{n}}\right)   \right]^{1/2}\Delta W_{n}^{1} \\
&P_{n+1}=P_{n}+P_{n}\left( -c+\dfrac{mN_{n}}{1+N_{n}}  \right)\Delta + \left[ P_{n} \left( c+\dfrac{mN_{n}}{1+N_{n}}  \right)  \right]^{1/2}\Delta W_{n}^{2},  \\
&N_{0}=N(0), \\ 
&P_{0}=P(0),
    \end{split}
 \end{equation} 
for $n=0,1,\ldots ,M-1$. 

We fix the initial condition $N(0)=1, P(0)=0.6$ once for all. For the corresponding deterministic predator-prey system (\ref{eq1.3}), we have calculated all three critical points 
$K_{1}=(0,0)$, $K_{2}=(k,0)$, 
and $K_{3}=\left(\dfrac{c}{m-c},\dfrac{k(m-c)-c}{k(m-c)^{2}}\right) $,
and identified two parameter regions for $m,c,k$: \begin{align*}
    &\Lambda _{1}=\left\{ (m,c,k):k>\dfrac{m+c}{m-c} \right\},\\  
    &\Lambda _{2}=\left\{ (m,c,k):k<\dfrac{m+c}{m-c} \right\}. 
 \end{align*}     
 We have shown in ~\cref{sec:dynamical} that, $K_{3}$ is a sink if $k<\dfrac{m+c}{m-c} $ and a source if $k>\dfrac{m+c}{m-c}$.

We choose two groups of parameter values so that in each parameter region 
$\Lambda _{1}$ and $\Lambda _{2}$, there is one and only one group of parameter values: 
\begin{equation}\label{eq5.3}
    \begin{split}
&\mathrm{Para}_{1}= \left\{m=3,c=1,k=3  \right\} \mbox{ corresponding to } \Lambda _{1}, \\ 
&\mathrm{Para}_{2}= \left\{m=3,c=1,k=1.5  \right\} \mbox{ corresponding to } \Lambda _{2}.
    \end{split}
 \end{equation}  
For each group of parameter values, we carry out a sufficient number of simulation runs to estimate expectation values $\mathrm{E}(N(t))$, $\mathrm{E}(P(t))$, 
and corresponding variances $\mathrm{Var}(N(t))$, $\mathrm{Var}(P(t))$, as functions of $t$. These estimates are calculated by the following estimator \cite{abundo1991stochastic}: 
\begin{align*}
&\hat{\mathrm{E}}(x(t))=\dfrac{1}{M ^{\prime} }\sum\limits_{k=1}^{M ^{\prime} }x^{k}(t), \\ 
&\hat{\mathrm{Var}}(x(t))=\dfrac{1}{M ^{\prime} }\sum\limits_{k=1}^{M ^{\prime} }\left[ x^{k}(t)-\hat{\mathrm{E}}(x(t)) \right]^{2}, 
 \end{align*}      
where $x$ represents $N$ and $P$, $M ^{\prime} $ is the number of simulation runs, and $x^{k}(t)$ represents, for every $t$, the approximated value of $x(t)$ obtained by means of the $k$-th simulaltion run.
In our simulations, we take $M =4\times 10^{4}$ and $M ^{\prime} =2\times 10^{4}$. 

When plotting the estimated expectation value $\mathrm{E}(N(t)) $, $\mathrm{E}(P(t)) $, we have also included the error bars for every $t$. To be specific, in these graphs, the middle curve is the estimated 
expectation as a function of $t$, and the upper and lower curve represent, respectively, \begin{align*}
\hat{\mathrm{E} }(x(t))\pm \dfrac{1}{2}\left[\hat{\mathrm{Var} }(x(t)) \right]^{1/2}, \quad x=N \mbox{ or } P.
 \end{align*}    

As illustrated in Figure~\ref{fig:fig1}, under the parameter setting $\mathrm{Para}_{1} $ in (\ref{eq5.3}), 
the trajectory of the estimated expectation of $N$ and $P$ approaches a final state indicated by the arrow. 
In the parameter region $\Lambda _{2}$, the trajectory of the estimated expectation of $N$ and $P$ approaches another final state as is demonstrated in Figure~\ref{fig:fig1}. 
The comparison between the deterministic model (\ref{eq1.3}) and the stochastic model (\ref{eq1.4}) 
demonstrates the discrepancy between the dynamics of the deterministic and stochastic systems and the large deviations of the estimated expectation of the population densities $N$ and $P$.

\begin{figure}[htbp]
    \centering
    \hspace*{-1.6cm} \includegraphics[width=1.1\textwidth]{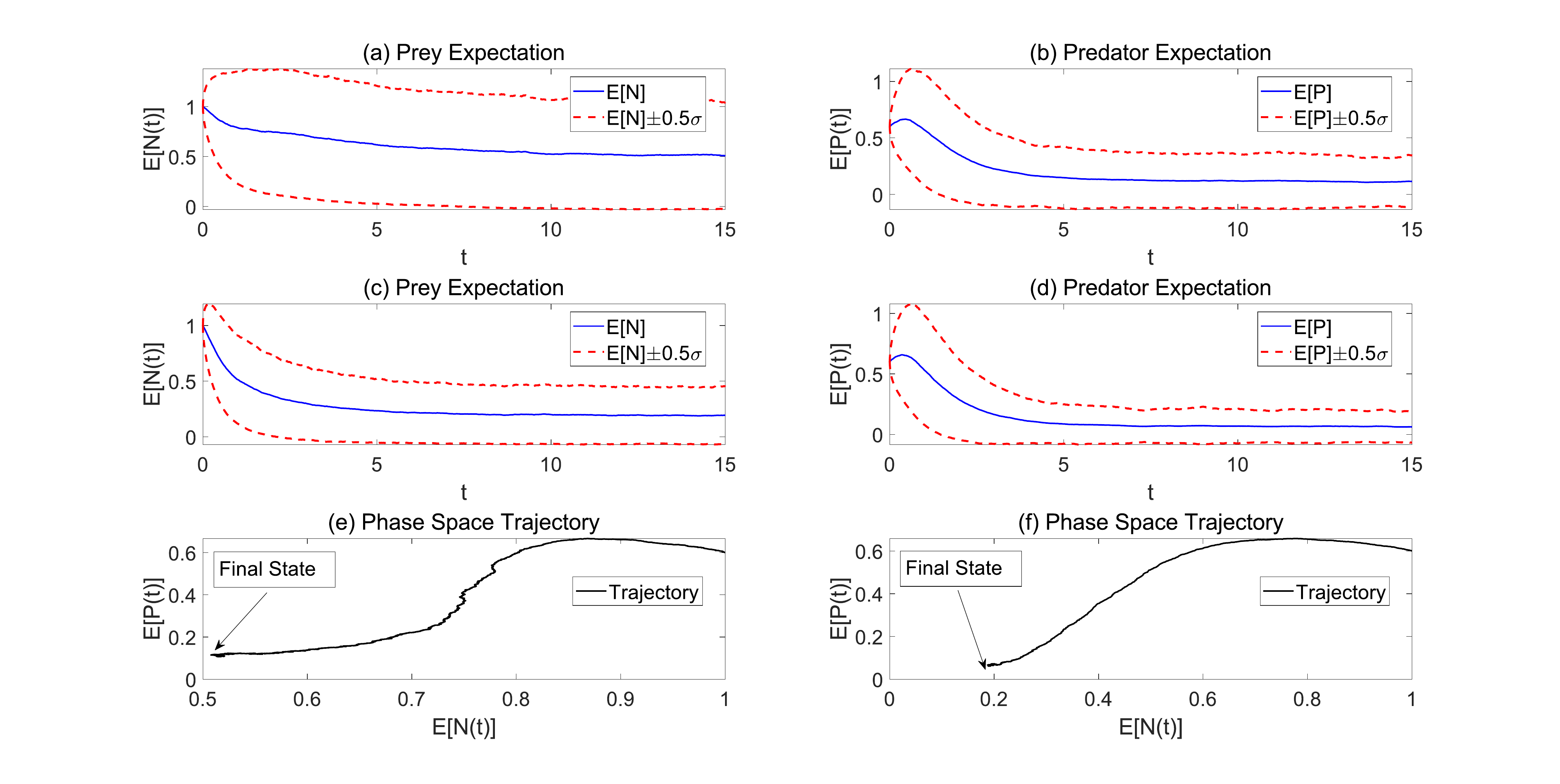}
    \caption{Simulation of estimated population density with error bars and the phase trajectory in the phase space. 
    Figure (a), (b), and (e) is the source equilibrium case with parameter value $\mathrm{Para}_{1} $ in (\ref{eq5.3}), 
    and a unique stable limit cycle exists for the system (\ref{eq1.3}) in this parameter setting. 
    Figure (a) is about $\hat{\mathrm{E}(N(t)) }$ as a function of $t$ (middle curve); the 
    upper and lower curve represent, respectively, $\hat{\mathrm{E} }(N(t))\pm \dfrac{1}{2}\left[\hat{\mathrm{Var} }(N(t)) \right]^{1/2}$; 
    figure (b) is the same thing for $P(t)$; figure (c) 
    is the trajectory $(\hat{\mathrm{E}}(N(t)),\hat{\mathrm{E} }(P(t)) )$ as a function of $t$. 
    Figure (c), (d), and (f) is the sink equilibrium case with parameter value $\mathrm{Para}_{2} $ in (\ref{eq5.3}),
    and a stable positive equilibrium exists for the system (\ref{eq1.3}) in this parameter setting. 
    Figure (c) is about $\hat{\mathrm{E}(N(t)) }$ as a function of $t$ (middle curve); the 
    upper and lower curve represent, respectively, $\hat{\mathrm{E} }(N(t))\pm \dfrac{1}{2}\left[\hat{\mathrm{Var} }(N(t)) \right]^{1/2}$; 
    figure (d) is the same thing for $P(t)$; figure (f) 
    is the trajectory $(\hat{\mathrm{E}}(N(t)),\hat{\mathrm{E} }(P(t)) )$ as a function of $t$.}
    \label{fig:fig1}
\end{figure}

\newpage
\section{Conclusion}\label{sec:conclusion}

We consider a stochastic Rosenzweig-MacArthur model. For the underlying deterministic system, we identify the stability property of three different steady states: 
the trivial steady state corresponding to the extinction of populations, the sole existence steady state of the prey population, 
and the co-existence steady state of the predator and prey population. 
These steady states have different stability properties, and the quantitative changes of parameter values may lead to a transition of the stability of equilibriums, 
indicating a bifurcation and a qualitative change in the dynamical properties of this ecological system.

There is much literature about the impacts of environmental randomness on the deterministic model and little research about the internal fluctuations
due to individuals' growth, reproduction, and deaths. We want to study this internal randomness; thus, the variations in our stochastic model are proportional to the population densities.
Next, we give a criterion for the existence and uniqueness result of SDE constrained on an open submanifold and use it to show that our stochastic system has a unique global strong solution. 
In the meantime, the invariance of the positive domain under our consideration is a biologically obvious result since population densities must be non-negative or positive; the invariance also allows us to restrict the behaviors of the coefficients of our SDE to submanifolds of the whole space $\mathbb{R} ^{n}$. Bearing this idea in mind, we obtain the estimation of the $p$-th moment, and the asymptotic estimation of the stochastic solution. 
The persistence of populations is a direct consequence of the invariance of the open domain and demonstrates that internal fluctuations in population densities do not have the same destabilizing effects as external perturbations and never lead to population extinction.
The asymptotic estimates indicate that the population will exhibit at most exponential growth, which aligns with the analysis results of these classical deterministic ecological models that in an ideal environment with sufficient resources (such as food), a population will experience exponential growth, while in a more realistic ecosystem, the population will undergo limited growth, with the maximum population density not exceeding the environmental carrying capacity due to resource competition.

Finally, we used the Euler-Maruyama scheme to run our simulations and used statistical estimators to calculate the expectations and variances of our population among different simulation runs. 
The numerical results show the discrepancy between the deterministic and stochastic systems and the large deviations of the estimated expectation of the population densities $N$ and $P$, indicating different dynamical properties between the deterministic and stochastic models.

\appendix
\section{Nondimensionalization for (\ref{eq1.2})}\label{app1}

We first have the Rosenzweig-MacArthur predator-prey model (\ref{eq1.2}):
\begin{align*}
        \dfrac{dN}{dt}&=rN(1-\dfrac{N}{K})-\dfrac{sNP}{1+s\tau N}, \\ 
        \dfrac{dP}{dt}&=-cP+d \dfrac{sNP}{1+s\tau N},
\end{align*}
Introducing two numbers $X>0$ and $Y>0$, which will be determined later, and considering the following transformations:
\begin{align*}
x=N/X, \quad y=P/Y,
 \end{align*}   
 we arrive at the following system:
 \begin{align*}
 \dfrac{dx}{dt} &= rx\left( 1-\dfrac{xX}{K}  \right) -\dfrac{sxyY}{1+s\tau xX}, \\ 
 \dfrac{dy}{dt} &= -cy + \dfrac{dsxyX}{1+s\tau xX}.   
  \end{align*} 
Setting $s\tau X = 1$, $Y=dX$, $d=\tau m$, and $k=\dfrac{K}{X} $,  we have 
\begin{align*}
 \dfrac{dx}{dt} &= rx\left( 1-\dfrac{x}{k}  \right) -\dfrac{mxy}{1+x}, \\ 
 \dfrac{dy}{dt} &= -cy + \dfrac{mxy}{1+x}.   
 \end{align*}   
Scaling the variables $t$ by $s=rt$, $m$ by $\bar{m}=\dfrac{m}{r} $, and $\bar{c}=\dfrac{c}{r} $   
we arrive at the final nondimensionalized system (\ref{eq1.3}) after dropping the bars:
\begin{align*}
\dfrac{dx}{ds} & = x\left( 1-\dfrac{x}{k}  \right) -\dfrac{\bar{m}xy}{1+x}, \\
\dfrac{dy}{ds} &= -\bar{c}y + \dfrac{\bar{m}xy}{1+x}.  
 \end{align*}

\newpage

\bibliographystyle{alpha}
\bibliography{sample}

\end{document}